\documentclass[12pt]{amsart}

\usepackage[centertags]{amsmath}
\usepackage{amsfonts}
\usepackage{amsthm}
\usepackage{newlfont}
\usepackage{amscd}
\usepackage{amsgen}
\usepackage{amssymb}
\usepackage{stmaryrd}
\usepackage{amssymb,amsmath}
\usepackage{amscd}
\usepackage{enumerate}

\newlength{\defbaselineskip} 
\theoremstyle{plain}
\newtheorem{thm}{Theorem}[section]
\newtheorem{cor}[thm]{Corollary}
\newtheorem{df}[thm]{Definition}
\newtheorem{lema}[thm]{Lemma}
\newtheorem{obs}[thm]{Proposition}
\newtheorem{exm}[thm]{Example}
\newtheorem*{tthm}{Main Theorem}
\newtheorem{pr}{Program}

\newtheorem{rem}[thm]{Remark}

\theoremstyle{definition} 
\theoremstyle{definition} \newtheorem{ex}{Example}[section] %

 \numberwithin{equation}{section}

\DeclareMathOperator{\rk}{rank}
\DeclareMathOperator{\diag}{diag}
\DeclareMathOperator{\sing}{Sing}

\DeclareMathOperator{\h}{ht}
\DeclareMathOperator{\Spec}{Spec}
\def\p{\mathbb{P}}

\def\ob{\begin{obs}}
\def\kob{\end{obs}}
\def\dow{\begin{proof}}
\def\kdow{\end{proof}}
\def\kwadrat{\hfill$\square$}
\def\tw{\begin{thm}}
\def\ktw{\end{thm}}
\def\lem{\begin{lema}}
\def\klem{\end{lema}}
\def\ex{\begin{exm}}
\def\prog{\begin{pr}}
\def\kprog{\end{pr}}
\def\wn{\begin{cor}}
  
\def\kwn{\end{cor}}
\def\uwa{\begin{rem}}
\def\kuwa{\end{rem}}
  
\def\kex{\end{exm}}
\def\dfi{\begin{df}}
\def\kdfi{\end{df}}
\begin{document}
\title[Birational maps between Calabi-Yau manifolds]{Birational maps between Calabi-Yau manifolds associated to webs of quadrics}
\author{Mateusz Micha\l ek}
\thanks{The author is supported by a grant of Polish MNiSzW (N N201 2653 33).}

\address{Mathematical Institute of the Polish Academy of Sciences,\linebreak ul. \'{S}%
niadeckich 8, 00-950 Warszawa, Poland}

\address{Unit\'{e} Mixte de Recherche 5582 CNRS --- Universit\'{e} Grenoble
1, \linebreak
100 rue des Maths, BP 74, 38402 St Martin d'H\`{e}res, France}
\email{wajcha2@poczta.onet.pl}

\maketitle
\begin{abstract}
We consider two varieties associated to a web of quad\-rics $W$ in $\p^7$. One is the base locus and the second one is the double cover of $\p^3$ branched along the determinant surface of $W$. We show that small resolutions of these varieties are Calabi-Yau mani\-folds. We compute their Betti numbers and show that they are not birational in the generic case. The main result states that if the base locus of $W$ contains a plane then in the generic case the two varieties are birational. 
\end{abstract}
\section*{Introduction}
In this paper we investigate the connection between two Calabi-Yau varieties associated to a web $W$ of quadrics in $\p^7$ ({\it{i.e.}} $W=\{\lambda_0Q_0+\dots+\lambda_3Q_3\mid(\lambda_0:\lambda_1:\lambda_2:\lambda_3)\in\p^3\}$, where $Q_i$ are linearly independent elements of $\mathcal O_{\p^7}(2)$). One of them is the base locus of the web. The second one is the double cover of $\p^3$ branched along the surface corresponding to degenerated quadrics of the system. We will consider generic webs and generic webs among webs containing a plane in the base locus.

The main result of the paper, where everything is defined over the field $\mathbb{C}$ of complex numbers (or any algebraically closed field of characteristic $0$), is the following theorem:
\begin{tthm}
Let $\textit W_P$ be a generic web of quadrics in $\p^7$ containing a given plane $P$. The base locus of $W_P$ and the double cover of $\p^3$ branched along the determinant surface corresponding to degenerated quadrics are birational varieties.
\end{tthm}
We will give a geometric description of the birational map and we will find correspondence between the singularities of the considered varieties (theorem \ref{10}). We also prove that both varieties admit small resolutions being Calabi-Yau manifolds. In case of a generic web (not containing a plane in the base locus) we show that the varieties have the same Euler characteristic, although they are not birational.

The following theorem was an inspiration for our work:
\tw
Let $X$ be a smooth intersection of two quadrics $Q_1$ and $Q_2$ in $\p^3$. Let $\p^1$ be the pencil spanned by $Q_1$ and $Q_2$. Then the variety $X$ is isomorphic to the double cover of $\p^1$ branched at the four points corresponding to the singular quadrics in the pencil. 
\ktw
This is a one dimensional analogue of our main theorem. The two dimensional case was investigated in \cite{NM} and \cite{Jan}.

In the first section we investigate the properties of the base locus $BS(W)$ of a web $W$ of quadrics in the projective space $\mathbb{P}^7$ and we recall a few general facts about webs of quadrics in projective spaces. We also describe the singularities of the base locus containing a plane and find a small resolution that is a Calabi-Yau manifold. We also compute some invariants of the considered varieties.

In the second section we describe the determinant surface associated to a web of quadrics. We also find some properties of the double cover of $\p^3$ branched along that surface.

The third section is the most important one. It combines the results of the first two parts to prove the main theorem.

\section{The base locus of a web of quadrics in $\mathbb{P}^7$}

In this section we will consider two cases:
\begin{enumerate}
\item a generic web of quadrics in $\mathbb{P}^7$,
\item a generic web of quadrics containing a fixed plane.
\end{enumerate}
We start with general remarks concerning webs of quadrics.
\ob\label{podstdim} The dimension of the family of webs of quadrics in $\p ^7$ equals 128.
\kob
\dow
Let $G(k,n)$ be a Grassmannian that parameterizes $k$-dimensional affine subspaces of an $n$-dimensional affine space.
Webs of quadrics in $\p^7$ are parameterized by $G(4,36)$ and $\dim G(4,36)=4\times 32=128$.
\kdow
We now compute the dimension of the family of webs of quadrics that contain a plane in the base locus. 
\ob \label{dim2} The dimension of the family of webs of quadrics in $\p^7$ that contain a plane in the base locus is equal to 119.
\kob
\dow
We consider $G(3,8)\times G(4,36)$. The first Grassmannian parameterizes planes in $\p ^7$, the second one webs of quadrics in $\p ^7$. Now, we consider the variety $S\subset G(3,8)\times G(4,36)$ defined by \[S=\{(p,s)\mid p\subset BS(s)\}.\] 
The fibers of the natural projection of $S$ onto the first coordinate are of dimension $\dim G(4,30)=104$.
By \cite[6.3, thm. 7]{Shaf} we get $\dim S=\dim G(3,8)+104=119$.
We now consider the second projection $q:S\rightarrow Q=q(S)\subset G(4,36)$, where $Q$ is the set of webs of quadrics that contain a plane in the base locus. We have already seen that $\dim Q\leq 119$. To see that $\dim Q=119$ we consider three varieties corresponding to three cases:
\begin{enumerate}
\item webs containing two disjoint planes in the base locus,
\item webs containing in the base locus two planes intersecting along a line,
\item webs containing in the base locus two planes intersecting at a point.
\end{enumerate}
One can easily check that the dimension of the corresponding varieties is respectively: 110, 114 and 111. This means that a generic web that contains a plane in the base locus, contains exactly one plane. We see that the generic fiber of $q$ has exactly one element, so $\dim Q=\dim S=119$.
\kdow
In the same way we can prove the following proposition:
\ob\label{dimcontaingivenplane}
The dimension of webs of quadrics in $\p^7$ containing a fixed plane $P$ in the base locus equals $\dim G(4,30)=104$.
\kob
By a repeated use of Bertini's theorem we obtain the following lemma:
\lem
Let $\emph W$ be a generic web of quadrics in $\p ^7$. Then the base locus $BS(W)$ is smooth.
\klem\kwadrat

However, the case of a generic web of quadrics that contain a common plane is more complicated. In the case of a net of quadrics in $\p ^5$ that contain a common line, the base locus could be smooth \cite{Jan}. In our case the following, general theorem proves that the intersection is always singular.
\tw\label{sing} If a projective variety $X\subset \p ^n$ contains a linear subspace $L$ of dimension $k$ and the defining ideal $I(X)=(f_1,\dots ,f_{n-2k+1})$, where $f_i$ are homogenous polynomials of the same degree, greater then 1, then at some point $x$ of $L$ the tangent space to $X$ is of dimension at least $2k$.
\ktw
\dow
We choose such a coordinate system that $L=\{(x_0:\dots :x_{n})\mid x_0=x_1=\dots =x_{n-k-1}=0\}$. 
Of course the tangent space to $X$ at any point $x\in L$ contains $L$. Let $f=(f_1,\dots ,f_{n-2k+1})$. The differential $d_xf$ is a matrix with $(n+1)$ columns and $(n-2k+1)$ rows. As the vector subspace $x_0=x_1=\dots =x_{n-k-1}=0$ of $\mathbb{C}^{n+1}$ is contained in the kernel of $d_xf$, the $(k+1)$ last columns of $d_xf$ are zero. Let $A$ be the matrix obtained from the first $n-k$ columns of $d_xf$. It is enough to prove that for some point $x\in L$, this matrix has rank less or equal to $n-2k$. The entries of $A$ are polynomials in $x_{n-k},\dots,x_n$ - depending on the choice of point $x$. These polynomials are of equal degree, say $d$. Let $l_i$ be the $i$-th row of matrix $A$. We want to prove that there exist $\lambda_j$, $1\leq j\leq n-2k+1$, not all equal to zero, such that $\sum_{j=1}^{n-2k+1}\lambda_j l_j=0$ for some $x_i$, $n-k\leq i\leq n$ not all equal to zero. 
These equations are of bidegree $(1,d)$. We have $n-k$ such equations. We may write them as $s_i=0$ for $1\leq i\leq n-k$, where $s_i\in \mathcal{O}_{\p ^{n-2k}\times \p ^k}(1,d)$. We know that $\mathcal{O}_{\p ^{n-2k}\times \p ^k}(1,d)$ is very ample and $\p ^{n-2k}\times \p ^k$ is of dimension $n-k$, so the zero-set of $n-k$ generic sections is not empty.
\kdow
Our next aim will be to compute how many singular points belong to the base locus of a generic system $\emph W_P$ spanned by four quadrics containing a fixed plane $P$. We start with the following lemma.

\lem \label{smooth}
The intersection of three generic quadrics containing a fixed plane $P$ is smooth.
\klem
\dow Let $Q_1$, $Q_2$ and $Q_3$ be three generic quadrics containing the plane $P$. Let $f=(Q_1,Q_2,Q_3)$. The singular points of the intersection satisfy: \[Q_1(x)=Q_2(x)=Q_3(x)=0\]\[\rk d_x f<3.\]
Due to the Bertini theorem it is enough to prove that such points do not exist on the plane $P$. Proceeding as in theorem \ref{sing}, we suppose that the plane $P$ is given by equations \[x_0=\dots=x_4=0.\] Let $A$ be the matrix obtained by choosing the first five columns of $d_xf$. This is a $5\times 3$ matrix with entries that are linear forms in $x_5$, $x_6$ and $x_7$. Let $l_1$, $l_2$ and $l_3$ be the rows of matrix $A$. We want to prove that in the generic case the equation \[\sum_{i=1}^3\lambda_i l_i=0\] does not have solutions in $((\lambda_1:\lambda_2:\lambda_3),(x_5:x_6:x_7))\in \p^2\times \p^2$. This equation may be written as intersection of five sections $s_j\in \mathcal{O}_{\p^2\times \p^2}(1,1)$, $1\leq j\leq 5$. Since $\dim\p^2\times\p^2=4$, the zero-set of five generic sections is empty.
\kdow
\lem\label{rrad}
The ideal $I$ generated by at most 4 generic quadrics in $\p^7$ that contain a plane $P$ is radical.
\klem
\dow
By Bertini's theorem the singular locus of the intersection is contained in $P$. We consider the following properties of a ring $A$\footnote{These are called Serre's conditions. We are using notation of \cite[17.1]{MatCA}.}:
 \begin{itemize}
\item the condition $(S_k)$ holds iff the depth of $(A_P)\geq\inf(k,\h P)$ for all prime ideals $P\in\Spec A$,

\item the condition $(R_k)$ holds iff the ring $A_P$ is regular for all ideals $P\in\Spec A$ such that $\h P\leq k$.
\end{itemize}
Using basic properties of Cohen-Macaulay rings one can prove that the ring $\mathbb C[X_0,\dots ,X_7]/I$ satisfies conditions $S_k$ for any $k$ and $R_0$, which is equivalent to being reduced and Cohen-Macaulay. 
\kdow
\begin{obs}\label{10osobliwych}
The intersection of 4 generic quadrics containing a common plane $P$ in $\p ^7$ has $10$ nodes lying on the plane $P$ as the only singularities.
\end{obs}
\begin{proof}
Due to Bertini's theorem it is enough to consider the singular points on the plane $P$.

We keep the notation of theorem \ref{sing}. We consider the variety in $\p^3\times \p^2$ defined by the ideal $I=(\sum_{j=1}^4\lambda_jl_j^1,\dots ,\sum_{j=1}^4\lambda_jl_j^4)$, where $l_j^i$ is the $i$-th entry of vector $l_j$. This ideal corresponds to the intersection of $5$ sections of type $(1,1)$, so it is enough to count $(H_1+H_2)^5=(^5_2)H_1^2H_2^3=10$ points, where $H_1$ and $H_2$ correspond respectively to a hypersurface of $\p ^2$ times $\p ^3$ and $\p ^2$ times a hypersurface of $\p ^3$. We obtain 10 points counted with multiplicity. This points are in the generic case distinct, because we intersected generic, very ample divisors. The fact that this points are nodes is a consequence of \cite{Kap}.
\end{proof}
\lem\label{prim}
The ideal $I$ generated by at most 4 generic quadrics in $\p^7$ that contain a plane $P$ is prime.
\klem
\dow
From the lemma \ref{rrad} we already now that this ideal is radical. Let $Y$ be an intersection of three generic quadrics. From the lemma \ref{smooth} we know that $Y$ is smooth. If $I$ was not prime, then the intersection of its components would have dimension at least $2$ that would contradict \ref{10osobliwych}.
\kdow
Now we will compute the Euler characteristic and Hodge numbers of the base locus.
\ob\label{eulersmooth}
Let $X$ be the base locus of a web $\textit{ W }$ of quadrics in $\p ^7$. The Euler characteristic $\chi (X)$ of $X$ is equal to -128.
\kob
\dow
The base locus $X$ is an intersection of four divisors $D_1,\dots ,D_4$, where $D_i$ is a divisor of a quadric that belongs to $\emph W$. Of course $\deg D_i=2$. Let $T_X$ be a tangent bundle of $X$. Using the formula \cite[3.2.12]{Fult} we get:
$$c(T_X)(1+2h)^4=(1+h)^{7+1},$$
where $c(T_X)$ is the total Chern class of the tangent bundle and $h$ is a hyperplane section on $X$. We have $h=c_1(i^*\mathcal{O}_{\p^7}(1))$ with $i:X\hookrightarrow \p^7$ is the inclusion of $X$ in $\p ^7$. We obtain $c_3(T_S)=-8h^3$. Now using \cite[3.2.13]{Fult} we get $\chi (S)=-8h^3\cap S$. Of course $\deg S=16$, so $\chi (S)=-8\times 16=-128$.
\kdow
\ob\label{hodgesmooth}
The Hodge numbers of Calabi-Yau manifolds that are complete intersections of four quadrics in $\p^7$ are respectivly:
\[h^{1,1}=1,\qquad h^{1,2}=65.\]
\kob
\dow
Let $X$ be a Calabi-Yau manifold given by the intersection of four quadrics in $\p^7$.
Using the adjunction formula exact sequence and Serre's duality we obtain:
$h^{1,2}=h^1(X,T_X)=65.$
\kdow

\subsection{Small resolution of $BS(W_P)$}

In our paper we use the following, algebraic definition of the Calabi-Yau manifold.
\dfi\label{dfCY}
Let $X$ be a smooth, $n$-dimensional, projective algebraic variety. We say that $X$ is a \emph{Calabi-Yau manifold} if and only if it satisfies the following conditions:
\begin{enumerate}
\item $K_X=0$,
\item$h^{i,0}=0$ for all $1\leq i<n$,
\end{enumerate}
where $K_X$ is the canonical divisor of $X$, $H^j(X,\Omega^k_X)$ is the $j$-th cohomology group of a sheaf of regular $k$-forms and $h^{k,j}=\dim H^j(X,\Omega^k_X)$.
\kdfi
Let $\textit W_P$ be a generic web of quadrics containing the plane $P$. Let $Y$ be the intersection of $3$ generic quad\-rics in the system. Let $\hat Y$ be the blow-up of $Y$ along the plane $P$. Let $\hat X$ be the birational transform of $X$.
\ob
The variety $\hat X$ is a Calabi-Yau manifold.
\kob
\dow
By proposition \ref{10osobliwych} the variety $\hat X$ is smooth. By the adjunction formula the canonical divisor $K_{\hat X}$ is equal to zero. Using the Lefschetz theorem and the exact sequence:
\[0\rightarrow \mathcal O_{\hat Y}(-\hat X)\rightarrow\mathcal O_{\hat Y}\rightarrow\mathcal O_{\hat X}\rightarrow 0\]
we see that $h^{0,1}_{\hat X}=0$, which proves the proposition.
\kdow
As a direct consequence of proposition \ref{10osobliwych} and proposition \ref{eulersmooth} we obtain the following result:
\ob
The Euler characteristic of $\hat X$ equals -108.
\kob


\section{Double cover of $\p ^3$ branched along the determinant surface}
Let $W$ be a web of quadrics in $\p^7$. We denote by $D_W\subset\p^3$ the locus of singular quadrics. If $W$ is sufficiently general then $D_W$ is a degree 8 surface given by: 
\[D_W=\{\lambda\in\p^3\mid\det(\lambda_1Q_1+\dots+\lambda_4Q_4)=0\},\]
where $Q_1,\dots,Q_4$ span $W$.
\lem\label{ranga}
For a generic web $\textit W$ of quadrics in $\p ^7$, singular points of $D_W$ correspond to quadrics of rank less or equal to 6 in the system $\textit W$.
\klem
\begin{proof}
We consider the $\p ^{35}$ that corresponds to the space of all symmetric matrices $8\times 8$. A generic web of quadrics corresponds to a generic three dimensional subspace of $\p ^{35}$. Let $O$ be the variety of quadrics of rank 7. Obviously $\dim O=34$. 

We first prove that for a generic system $\textit W$, the determinant octic $D_W$ does not have singular points that correspond to quadrics of rank 7. Due to lemma \ref{prim} it is enough to consider only systems that generate a radical ideal.

{\it{Step 1:}} Let us fix $Q_0\in O$. We answer now a following question: for which systems $\textit W$ of quadrics generating a radical ideal is $Q_0$ a singular point of $D_W$? We choose a coordinate system such that $Q_0=\diag (1,1,1,1,1,1,1,0)$. We may choose quadrics $Q_1$, $Q_2$ and $Q_3$ such that the system $W$ is spanned by them and the quadric $Q_0$. Let $Q_{\lambda}=\lambda_0 Q_0+\lambda_1 Q_1+\lambda_2 Q_2+\lambda_3 Q_3$ for $\lambda=(\lambda_0:\dots:\lambda_3)\in\p ^3$. Let $\lambda'=(1:0:0:0)$. Then the point $Q_0$ is a singular point of $D_W$ if and only if the following equations hold:
\begin{equation*}
{\frac{\partial}{\partial\lambda_0}\det Q_{\lambda}}_{|\lambda=\lambda'}=0
\end{equation*}
\begin{equation}\label{eq}
{\frac{\partial}{\partial\lambda_1}\det Q_{\lambda}}_{|\lambda=\lambda'}=0
\end{equation}
\begin{equation*}
{\frac{\partial}{\partial\lambda_2}\det Q_{\lambda}}_{|\lambda=\lambda'}=0
\end{equation*}
\begin{equation*}
{\frac{\partial}{\partial\lambda_3}\det Q_{\lambda}}_{|\lambda=\lambda'}=0.
\end{equation*}
Let $v_j^{(i)}$ be the $j$-th entry of the last column of matrix $Q_i$ and let $(Q_\lambda)_i$ be the $7\times 7$ matrix obtained from $Q_\lambda$ by deleting the $i$-th row and the last column. Using the Laplace formula we obtain:
\[{\frac{\partial}{\partial\lambda_i}\det Q_{\lambda'}}_{|\lambda=\lambda'}=\sum_{j=0}^7(-1)^{j+1}v_j^{(i)}\det (Q_{\lambda'})_i+\]\[+\sum_{j=0}^7(-1)^{j+1}(0\times v_j^{(1)}+0\times v_j^{(2)}+0\times v_j^{(3)})\frac{\partial}{\partial\lambda_i}\det (Q_\lambda)_i=v_7^{(i)},\]
because $\det (Q_{\lambda'})_i=0$ for $i\neq 7$ and $\det (Q_{\lambda'})_i=1$ for $i=7$. This means that the system of equations \ref{eq} is equivalent to $v_7^{(i)}=0$ for $i=1,\dots,4$. This condition holds if and only if all matrices in the system $\textit W$ have 0 as the entry in the lower right corner.

All such matrices correspond to a projective hyperplane in $\p^{35}$. The set of webs of such quadrics can be parameterized by a Grassmannian $G(3,34)$ (3 vectors and $Q_0$ form a base of such a web). We see that the dimension of all webs for which there exists a quadric of rank $7$ that is a singular point of $D_W$ is at most $\dim O+\dim G(3,34)=34+93=127$. This means that for a generic web of quadrics, the quadrics of rank $7$ are not singular points of $D_W$.

{\it{Step 2:}} Now we prove, that for a system $\textit W$, matrices of rank less or equal to 6 are always singular points of $D_W$ (c.f. \cite[lemma 2.4]{Jan}).
Let $Q_0$ be a quadric in the system of rank less or equal to 6, $Q_0=\diag(1,\dots,0,0)$. We keep the rest of the notation. We obtain:
\[{\frac{\partial}{\partial\lambda_i}\det Q_{\lambda'}}_{|\lambda=\lambda'}=\sum_{j=0}^7(-1)^{j+1}v_j^{(i)}\det (Q_{\lambda'})_i+\]\[+\sum_{j=0}^7(-1)^{j+1}(0\times v_j^{(1)}+0\times v_j^{(2)}+0\times v_j^{(3)})\frac{\partial}{\partial\lambda_i}\det (Q_\lambda)_i=0,\]
because $\det(Q_{\lambda'})_i=0$.
\end{proof}
\tw\label{84}
For a generic web $\textit W$ the surface $D_W$ has exactly 84 singular points.
\ktw
\dow
From the lemma \ref{ranga} it is enough to check how many matrices of rank less or equal to $6$ there are in a generic web of quadrics. Once again we consider $\p ^{35}$ that correspond to the space of all $8\times 8$ matrices. The variety $M$ of the matrices of rank less or equal to 6 correspond to the zeros of ideal $i$ generated by all $7\times 7$ minors of $m$, where:
\[
m=\left[
\begin{array}{cccccccc}
x_0&x_1&x_2&x_3&x_4&x_5&x_6&x_7\\
x_1&x_8&x_9&x_{10}&x_{11}&x_{12}&x_{13}&x_{14}\\
x_2&x_9&x_{15}&x_{16}&x_{17}&x_{18}&x_{19}&x_{20}\\
x_3&x_{10}&x_{16}&x_{21}&x_{22}&x_{23}&x_{24}&x_{25}\\
x_4&x_{11}&x_{17}&x_{22}&x_{26}&x_{27}&x_{28}&x_{29}\\
x_5&x_{12}&x_{18}&x_{23}&x_{27}&x_{30}&x_{31}&x_{32}\\
x_6&x_{13}&x_{19}&x_{24}&x_{28}&x_{31}&x_{33}&x_{34}\\
x_7&x_{14}&x_{20}&x_{25}&x_{29}&x_{32}&x_{34}&x_{35}\\
\end{array}
\right].
\]
Using the program Singular \cite{Singular} we see that the variety $M$ is of dimension 32 and degree 84 (program \ref{l1}). This means that a generic three dimensional hyperplane has exactly 84 points in common with $M$, which proves the theorem.
\kdow
\prog\label{l1}
\begin{verbatim}

ring r=0,x(1..36),dp;
matrix m[8][8]=x(1..8),x(2),x(9..15),x(3),x(10),x(16..21),
x(4),x(11),x(17),x(22..26),x(5),x(12),x(18),x(23),
x(27..30),x(6),x(13),x(19),x(24),x(28),x(31..33),
x(7),x(14),x(20),x(25),x(29),x(32),x(34),x(35),x(8),x(15),
x(21),x(26),x(30),x(33),x(35),x(36);
ideal i=minor(m,7);
i=std(i);
degree(i);
//dimension(proj.) = 32;
//degree(proj.) = 84;
\end{verbatim}
\kprog
\uwa
Due to lemma \ref{ranga} theorem \ref{84} is a consequence of \cite[prop. 12 b)]{Tu}, because the number of singular points is equal to:
\[\prod_{a=0}^1\frac{(^{8+a}_{2-a})}{(^{2a+1}_{\hskip 10pt a})}=84.\]
\kuwa
\begin{rem}\label{oszust}
One can easily check that the points described in theorem \ref{84} are nodes. It is enough to look at all quadrics of rank 7 in the neighborhood of a point corresponding to a quadric of rank 6. 
\end{rem}

\begin{rem}\label{84rang7pl}
Using the same arguments as above we can prove that for a generic web $W_P$ of quadrics containing a plane, there are $84$ quadrics of rank $6$ that are singular points of $D_{W_P}$.
\end{rem}
However, as we will prove later, in the case of remark \ref{84rang7pl} there are 10 more singular points of $D_W$ that correspond to some special matrices of rank 7.
\lem \label{si} Let $\textit W_P$ be a generic web of quadrics containing a fixed plane $P$. Then all quadrics of rank less or equal to 6 that belong to $\textit W_P$ do not have singular points on $P$.
\klem
\dow
We choose such a system of coordinates that \[P=\{(x_0;\dots ;x_7)\in\p ^7\mid x_0=\dots=x_4=0\}.\] Let $m\in \textit W_P$ be an $8\times 8$ matrix. Let $l_i$ be the $i$-th column of $m$. Of course in a generic case the ideal generated by the quadric $m$ is radical. We see that $m$ has a singular point on $P$ if and only if $l_6$, $l_7$ and $l_8$ are linearly dependent. Using Singular \cite{Singular} we can prove that a generic system $\textit W_P$ does not contain matrices of rank less or equal to 6 that have three last columns dependent (program \ref{l2}). The theorem follows.
\kdow
\prog\label{l2}
\begin{verbatim}

ring r=0,x(1..30),dp;
matrix m[8][8]=x(1..8),x(2),x(9..15),x(3),x(10),x(16..21),
x(4),x(11),x(17),x(22..26),x(5),x(12),x(18),x(23),
x(27..30),x(6),x(13),x(19),x(24),x(28),0,0,0,x(7),x(14),
x(20),x(25),x(29),0,0,0,x(8),x(15),x(21),x(26),x(30),0,0,0;
matrix n[3][8]=m[5..8,1..8];
ideal i=minor(m,7);
ideal j=minor(n,3);
i=j,i;
i=std(i);
degree(i);
//dimension(proj.) = 25;
//degree(proj.) = 5;
\end{verbatim}
\kprog
\tw\label{10}
If $\textit W_P$ is a generic web of quadrics that contains a plane $P$ in its base locus then $D_{\textit W_P}$ has 94 singular points: 84 that correspond to quadrics of rank less or equal to 6 and 10 that correspond to quadrics of rank 7 that have singularities on the plane $P$.
\ktw
\dow
Due to remark \ref{84rang7pl} is enough to prove that there are 10 singular points that belong to $D_{\textit W_P}$ that correspond to matrices of rank 7. We will associate with each such quadric, exactly one singular point of $BS(\textit W_P)$.
Let $a_1,\dots, a_{10}$ be $10$ singular points of $BS(\textit W_P)$ described in \ref{10osobliwych}. Let $Q_1,\dots ,Q_4$ be the generators of the system $\textit W_P$ and let $f_i$ be the equation of $Q_i$. We know that the matrix $\frac{\partial f_i}{\partial x_j}(a_k)$, $1\leq i\leq 4$, $0\leq j\leq 7$ has rank less or equal to 3 for each $1\leq k\leq 10$. That means that for each point $a_k$ there exist $\lambda_1,\dots ,\lambda_4$ such that: \[\sum_{i=1}^4\lambda_idf_i(a_k)=0.\] Of course this equality tells us that $a_k$ is a singular point of $g_k=\lambda_1f_1+\dots +\lambda_4f_4$. The quadric $g_k$ has a singularity on the plane $P$, so for a generic system $\textit W_P$, it is of rank 7 due to lemma \ref{si}. In a suitable system of coordinates we may assume that $g_k=\diag(1,1,1,1,1,1,1,0)$ and $a_k=(0,\dots , 0,1)$. The point $a_k$ belongs to all members of the system and that means that in this system of coordinates other matrices in $\textit W_P$ have zero as the last entry. This tells us, as already stated in lemma \ref{ranga}, that $g_k$ is a singular point of $D_{\textit W_P}$. On the other hand if some $g$ is a singular point of $D_{\textit W_P}$ and $g$ is of rank 7, then one easily sees that the singular point of $g$ lies on the base locus of the system (compare also with the proof of the lemma \ref{ranga}) and so it is a singular point of this variety. Of course, as we have already proved in proposition \ref{10osobliwych} the numbers $\lambda_1,\dots ,\lambda_4$ are, in a generic case, unique for each $a_k$, which proves the theorem.
\kdow
\lem
A nonsingular surface $S$ of degree 8 in $\p ^3$ has Euler characteristic 304.
\klem
\dow
Once again we use the example \cite{Fult} to see that $c_2(T_S)=38h^2$, where $h$ is a class of a hyperplane section and $T_S$ is the tangent bundle of $S$. This means that the Euler characteristic of $S$ is equal to $\deg (38 h^2\cap S)=8\times 38=304$.
\kdow
\lem\label{z}
The double cover $C$ of $\p ^3$ branched along a smooth surface $S$ of degree 8 has the Euler characteristic $\chi(C)=-296$.
\klem
\dow
It is well known that the Euler characteristic is additive for algebraic sets over $\mathbb C$, and $\chi (\p ^n)=n+1$, so $\chi (\p^3\backslash S)=4-304=-300$. We obtain $\chi (C)=2\times\chi(\p^3\backslash S)+\chi (S)=-600+304=-296$.
\kdow
\ob\label{eulerpdnakgen}
Let $W$ be a generic web of quadrics in $\p^7$. Let $Z$ be the double cover of $\p^3$ branched along the determinant surface $D_W$. Then the Euler characteristic $\chi(Z)=-212$.
\kob
\dow
Let $C$ be as in lemma \ref{z}. Using \cite[cor. 4.4]{D} and remark \ref{oszust} we obtain:
\[\chi(Z)=\chi(C)+84=-212.\]
\kdow
\wn
Any small resolution $\widehat Z$ of the variety $Z$ described in \ref{eulerpdnakgen} has Euler's characteristic $\chi(\widehat Z)=\chi(Z)+84=-128$.
\kwadrat\kwn
Analogously for a generic web containing a plane in the base locus:
\wn\label{eulernakroz}
A small resolution $\widehat {Z'}$ of the double cover $Z'$ of $\p^3$ branched along the determinant surface has Euler's characteristic $\chi (\widehat {Z'})=-108$.
\kwn\kwadrat
\ob\label{hodge}
The Hodge numbers of any small resolution $\widehat {Z'}$ are respectively:
\[h^{1,1}=2,\qquad h^{1,2}=56.\]
\kob
\dow
We can compute $h^{1,1}$ using theorems on defects \cite{Cynk1, Cynk2} and then $h^{1,2}$ from corollary \ref{eulernakroz}.
\kdow

\section{The correspondence variety and the birationality of described constructions}
In this section our aim is to prove the main theorem:
\begin{thm}\label{glowne}
Let $\textit W_P$ be a generic web of quadrics containing a plane $P$. Then the base locus of $W_P$ and the double cover of $\p^3$ branched along the determinant surface corresponding to degenerated quadrics are birational varieties.
\end{thm}
Let us give the main ideas of the proof.

Let $Q$ be a generic quadric of the web $W_P$. We will show that it contains two 3-dimensional projective spaces that contain the plane $P$ (lemma \ref{zaw1} below). Each of these projective 3-dimensional spaces intersects the base locus $BS(W_P)$ at $P$ (of course) and at one more point (all possibilities are described in the lemma \ref{przec} and the proof for the generic quadric uses a few technical lemmas \ref{wykl}, \ref{proste}, \ref{nolineonP}, \ref{nodoubleP}, \ref{dwadojeden} and \ref{jedendojeden}). 

We also prove that a generic degenerated quadric contains one 3-dimensional projective space that contains the plane $P$ (lemma \ref{jednaw7}). This projective 3-dimensional space intersects the base locus $BS(W_P)$ also at $P$ and at one more point.
Conversely, it turns out (lemma \ref{punktkwadryke}) that generically a point of the base locus of $W_P$ also determines a unique quadric.
This gives us the birational map involved in the theorem.

We will start by recalling some facts about linear subspaces contained in a quadric.
\lem\label{zaw1}
Let $Q$ be a non-degenerated quadric in $\p^{2n+1}$ that contains a vector subspace $P$ of projective dimension $n-1$. Then $Q$ contains exactly two vector subspaces of projective dimension $n$ that contain $P$.
\kwadrat\klem
\lem\label{jednaw7}
Let $Q$ be a quadric in $\p ^7$ of rank $7$ that contains a plane $P$. If the singular point of $Q$ is not on $P$, then $Q$ contains exactly one linear subspace of dimension $3$ that contains $P$.
\klem
\dow
Let $R$ be the singular point of $Q$. By choosing a suitable system of coordinates we may suppose that $Q=\diag(1,1,1,1,1,1,1,0)$. Of course $R=(0:\dots :0:1)$. Let $L$ be the hyperplane given by $x_7=0$. Since $R\not\in P$ if we project $Q$ from $R$ on $L$ we obtain a smooth quadric that contains the projection $P'$ of $P$. One can easily see that $M=\{(x_0:\dots :x_7)\mid (x_0:\dots :x_6)\in P'\}$ is a hyperplane of dimension $3$ contained in $Q$, that contains $P$. If there existed another hyperplane $M'$ with such a property then we would be able to choose $x\in M'\setminus M$. Let $x'$ be the projection of $x$ onto $L$. From definition, $x'\not\in P'$. The projection of $M'$ onto $L$ would be a hyperplane that contains $P'$ and $x'$, so would be of dimension at least $3$. The theorem \ref{sing} shows that in this situation the projection of $Q$ would be singular, which is not true.
\kdow
\uwa\label{sty}
The sum of subspaces described in lemmas \ref{jednaw7} and \ref{zaw1} for a quadric $Q$ equals $Q\cap \bigcap_{p\in P}T_pQ,$ where $T_pQ$ is the tangent space at $p$ to $Q$.
\kuwa
\lem\label{przec}
Let $\textit W_P$ be a web of quadrics containing a plane $P$ in its base locus. Let $C$ be a three dimensional subspace that contains $P$. Then the scheme theoretic intersection $BS(\textit W_P)\cap C$ is one of the following:

1) the plane $P$ and a point outside $P$,

2) the plane $P$ and a point on it (exact explanation in the proof)

3) the plane $P$ and a line intersecting it properly,

4) the plane $P$ and a line on it,

5) double plane $P$,

6) two planes (one of them is of course $P$),

7) the whole $\p^3$,

8) the plane $P$.
\klem
\dow
Let $\textit W_P$ be spanned by four quadrics $Q_1,\dots ,Q_4$. Let $Q_i'$ be the restriction of each quadric $Q_i$ to $C$. The restriction of each quadric $Q_i$ contains the plane $P$, so each restriction is either zero or has two components: planes, one of which is $P$. Let $P_i$ be the second one. Of course the base locus of $\textit W_P$ restricted to $C$ is the intersection of $Q_i'$. This is the plane $P$ and the intersection of (at most) four planes $P_i$. The lemma describes all the possibilities of this intersection.
\kdow
\ob\label{wykl}
For a generic system $\textit W_P$ cases 6 and 7 of the lemma \ref{przec} do not occur.
\kob
\dow
In those cases the system would contain two different planes in its base locus and this is not the case as we showed in the proof of theorem \ref{dim2}.
\kdow
We keep the notation of lemma \ref{przec}.
\ob\label{stycz}
Let $\textit W_P$ be a generic web of quadrics containing a plane $P$. Let $o\in P$ be a point of the intersection of four planes $P_i$. Then the linear subspace $C$ is contained in the tangent space $T_o(BS(\textit W_P))$ to the variety $BS(\textit W_P)$ at the point $o$. If $o$ is a smooth point of $BS(\textit W_P)$ then $C=T_o(BS(\textit W_P))$.
\kob
\dow
The linear subspace $C$ is contained in the tangent space at $o$ to each quadric $Q_i$, so is also contained in the tangent space of the base locus. The second part follows by comparing dimensions.
\kdow
\lem\label{proste}
A generic system $\textit W_P$ contains in its base locus only a finite number of lines that intersect the plane $P$ properly.
\klem
\dow
Let $M \subset G(3,8)\times G(2,8)$ be defined by:
\[M=\{(p,l)\mid\#(p\cap l)=1\}.\]
The variety $M$ corresponds to planes and lines intersecting properly in $\p^7$.
It is easy to check that $\dim M=23$.

Let $N$ be the subvariety of $M\times G(4,36)$, where $G(4,36)$ parameterizes webs of quadrics in $\p^7$ defined by: \[N=\{((p,l),s)\mid (p,l)\in M;\ p,l\subset BS(s)\}.\]
Let $p_1$ be the projection of $N$ onto the first coordinate. The fiber of $p_1$ is isomorphic to $G(4,28)$ that is of dimension 96, so $\dim N=119$. Let $Q$ be the subvariety of $G(4,36)$ that corresponds to the systems that contain a plane in its base locus. From theorem \ref{dim2} we know that dim $Q=119$. Let $p_2:N\rightarrow Q$ be the projection of $N$ onto the second coordinate. We see that the generic fiber of $p_2$ is of dimension at most 0, so a generic system that contains a plane $P$, contains a finite number of lines that intersect $P$ properly.
 \kdow

\subsection{The incidence variety S}

For a generic system $\textit W_P$ we will construct a variety $S\subset BS(\textit W_P)\times \p^3$, where $\p^3$ parameterizes all quadrics in the system $\textit W_P$. We will use the same notation for a point in $\p^3$ and a quadric corresponding to it. 
Let \[S_0=\{(p,q)\mid p\in q\cap\bigcap_{s\in P}T_s q\}\]
and let
\[S=\overline{S_0\backslash (P\times\p^3)}.\]
Due to remark \ref{sty} the variety $S_0$ corresponds to quadrics and points that appear as an intersection of $BS(\textit W_P)$ with the linear subspaces of dimension 3 that contain $P$ and are contained in a given quadric. Such an intersection always contains the plane $P$ and other points described in lemma \ref{przec}, that are most important for us. The variety $S$ is obtained from $S_0$ by removing a component corresponding to intersection points on $P$.

\dfi\label{defA}
For a given web $W_P$ we consider the set of all 3 dimensional subspaces containing $P$ and contained in some member of the web. Let $E$ be the subset of this set that consists of such subspaces that intersect $BS(W_P)$ in $P$ and a line (situations 3) and 4) of lemma \ref{przec}).
We define $A$ to be the set of all these lines contained in $BS(W_P)$ that appear as an intersection of some member of $E$ with $BS(W_P)$.
\kdfi
\lem\label{nolineonP}
For a generic web $\textit W_P$ there are no lines from the set $A$ that are contained in $P$.
\klem
\dow
We will bound the dimension of the set $Z\subset G(4,36)$ of systems that contain $P$ in the base locus and for which there is a line in $A$ that lies on $P$. For this purpose we consider a variety $Q\subset G(2,3)\times G(1,5)\times G(4,36)$, where $G(2,3)$ parameterizes lines on $P$, $G(1,5)$ parameterizes three dimensional projective subspaces containing $P$ in $\p^7$ and $G(4,36)$ parameterizes webs of quadrics in $\p^7$. The variety $Q$ is defined as follows:
\[Q=\{(l,V,s)\mid l\subset P\subset BS(s), V\cap BS(s)=P\cup l\text{ and } \exists{k\in S}:V\subset k\}.\]
Let $q:Q\rightarrow G(2,3)\times G(1,5)$ be the projection of $A$ onto the first two coordinates. We consider the fiber of $q$ above $(l,V)$. We know that each quadric of the system in the fiber restricted to $V$ is either $V$ or defines two planes (one of them is $P$) intersecting in $l$. We may assume that:
\[P=\{(x_0:\dots:x_7)\mid x_0=\dots=x_4=0\},\]
\[l=\{(x_0:\dots:x_7)\mid x_0=\dots=x_5=0\},\]
\[V=\{(x_0:\dots:x_7)\mid x_0=\dots=x_3=0\}.\]
This means that the matrix corresponding to this quadric is of the form:
\[
\left[
\begin{array}{cccccccc}
*&*&*&*&*&*&*&*\\
*&*&*&*&*&*&*&*\\
*&*&*&*&*&*&*&*\\
*&*&*&*&*&*&*&*\\
*&*&*&*&*&*&0&0\\
*&*&*&*&*&0&0&0\\
*&*&*&*&0&0&0&0\\
*&*&*&*&0&0&0&0\\
\end{array}
\right]
\]
All such matrices form a 28 dimensional projective subspace, so the dimension of the fiber is isomorphic to $G(4,28)$ and is of dimension 96. This means that $\dim Q\leq 96+\dim G(2,3)+\dim G(1,5)=96+2+4=102$. Let $\pi: Q\rightarrow Z$ be the projection of $Q$ onto the last coordinate. Of course this is a surjection, so $\dim Z\leq 102$. Due to proposition \ref{dimcontaingivenplane} this means that the dimension of $Z$ is strictly lower then the dimension of the set of webs that contain $P$ in the base locus. This proves the theorem.
\kdow
\wn\label{nodoubleP}
A generic web $\textit W_P$ does not contain a quadric $q$ such that a three dimensional projective space $V$ is contained in $q$ and $V$ intersected with $BS(\textit W_P)$ is the double plane $P$ (case 5 of the lemma \ref{przec}).
\kwn
\dow
Taking any line $l$ belonging to $P$ we see that such a system would belong to $Z$. The corollary follows.
\kdow
\lem\label{skonpr}
For a generic web $\textit W_P$, the set $A$ is finite.
\klem
\dow
Due to lemma \ref{proste} there is only a finite number of lines that intersect $P$ properly and due to lemma \ref{nolineonP} there are no other lines.
\kdow

\lem\label{dwadojeden}
Let $q$ be a generic non-degenerated quadric belonging to a generic web $\textit W_P$. There are precisely two points $p_1$ and $p_2$ such that $(p_i,q)\in S$ for $i=1,2$. Each point $p_i$ is a smooth point of $BS(\textit W_P)$. \klem
\dow
From the lemma \ref{zaw1} we know that $q$ determines exactly 2 linear subspaces $T_1$ and $T_2$ that contain $P$.  Let $A$ be the set defined in \ref{skonpr}.

The set $A$ is finite (\ref{skonpr}) and contains only lines that intersect $P$ properly (\ref{nolineonP}), so case 4) from lemma \ref{przec} is impossible. Each line in $A$ contains a smooth point $x$ of $BS(\textit W_P)$, so by \ref{stycz} it determines just one three dimensional projective space (spanned by $P$ and $x$). Let $A'$ be the set of those linear subspaces that contain $P$, are contained in some quadric in the system and intersect $BS(\textit W_P)$ in the plane $P$ and a line (case 3 of the lemma \ref{przec}). 

Let $l\in A$ be a line and $L\in A'$ the corresponding projective subspace. The line $l$ is equal to the intersection of the components $P_i$ different from $P$ of the restrictions of quadrics $Q_1,\dots, Q_4$ that span $W_P$ to the subspace $L$. This means that exactly two of the planes $P_i$ are linearly independent. We may assume that:
\[P_3=aP_1+bP_2,\]
\[P_4=cP_1+dP_2.\]
We see that a quadric $q=\sum_{i=1}^4\lambda_i Q_i$ contains $L$ if and only if:
\[\lambda_1P_1+\lambda_2P_2+\lambda_3(aP_1+bP_2)+\lambda_4(cP_1+dP_2)=0,\]
that is equivalent to:
\[\lambda_1+\lambda_3a+\lambda_4c=0,\]
\[\lambda_2+\lambda_3b+\lambda_4d=0.\]
These equations determine a line in $\p^3$. This means that to each subspace in $A'$ corresponds exactly one line in $\p^3$. The set $A'$ is finite, so for a generic quadric $q\in W_P$ the subspaces $T_1$ and $T_2$ do not belong to $A'$.

By what we just proved and due to corollary \ref{nodoubleP} and proposition \ref{wykl} the only possibilities left are cases 1), 2) and 8) of lemma \ref{przec}. Let $q=\sum_{j=1}^4\lambda_j Q_j$. The plane $T_i$ belongs to $q$, so $\sum_{j=1}^4\lambda_j Q_j'=0$. Let $f_i$ be the equation of $P_i$. We know that $\sum_{j=1}^4\lambda_j f_j=0$, so the subspaces $P_i$ are linearly dependent. This means that their intersection is not empty, so the only possibilities are cases 1) and 2). We can therefore define $p_i$ to be the points that are components different from $P$ of the intersection of $T_i$ with $BS(W_P)$. First we will prove that in a generic situation each $p_i$ is smooth, then that $p_1\neq p_2$.
We want to see for which quadrics the intersection of $T_i$ with $BS(\textit W_P)$ is (as a component different from $P$) one of the ten singular points described in \ref{10osobliwych}. Suppose that $a$ is a singular point of $BS(\textit W_P)$. Due to lemma \ref{smooth} the tangent space $T_a(BS(\textit W_P))$ is of dimension 4, so the linear subspace $T_i$ contains $P$ and is contained in some fixed 4 dimensional linear subspace thanks to \ref{stycz}. The dimension of the set of all such linear subspaces equals $\dim G(1,2)=1$. Each such subspace determines a quadric uniquely, so the set of quadrics for which $p_i$ is singular is of dimension one.

To prove that $p_1\neq p_2$ it is enough to use \ref{stycz} and notice that a quadric determines two \emph{different} linear subspaces as proved in \ref{zaw1}.
\kdow

\lem\label{jedendojeden}
Let $\textit W_P$ be a generic system containing a plane $P$ in its base locus. For a generic quadric $q$ of rank $7$ there exists exactly one point $p$ such that $(p,q)\in S$.
\klem
\dow
The quadrics of rank $7$ correspond to an open subset of the determinant surface $D_{W_P}$. We may apply lemma \ref{jednaw7}, because there are only 10 quadrics that have a singular point on $P$, such a point is also a singular point of $BS(\textit W_P)$. We can repeat the proof of lemma \ref{dwadojeden}, because each time we removed a one dimensional subset, which proves the theorem.
\kdow
\lem\label{punktkwadryke}
For a generic system $\textit W_P$ let $L$ be the sum of all lines in the set $A$ defined in \ref{defA}.
For any point $p\in BS(\textit W_P)\backslash (L\cup \sing(BS(\textit W_P)))$ there is exactly one quadric $q$ such that $(p,q)\in S$.
\klem
\dow
First we show the existence.
Let $Q_1,\dots,Q_4$ be quadrics that span $\textit W_P$. Let $C$ be a linear subspace of dimension 3 spanned by $P$ and $p$ if $p\not\in P$ or $C=T_p (BS(W_P))$ if $p\in P$. Let $g_i$ be the equation of the restriction of $Q_i$ to $C$. We get that $g_i=f\times f_i$ where $f$ defines $P$ and $f_i$ defines another plane. We know that $(f_1,\dots,f_4)$ defines a point, so $f_i$ are linearly dependent. Suppose $\sum_{i=1}^4\lambda_i f_i=0$. Of course the quadric $Q=\sum_{i=1}^4\lambda_i Q_i$ satisfies the conditions of the theorem.

Now we show the uniqueness. If there existed two quadrics $Q_1$ and $Q_2$ that contained $C$, we would be able to choose $Q_3$ and $Q_4$ such that $Q_1,\dots,Q_4$ would span $\textit W_P$. The intersection of the $BS(W_P)$ with $C$ would be the intersection of the restriction of $Q_3$ and $Q_4$ to $C$. This is the plane $P$ and an intersection of two hyperplanes. This of course cannot be the plane $P$ and the point $p$. The contradiction proves the theorem.
\kdow
\dow[Proof of theorem \ref{glowne}]
Due to lemmas \ref{dwadojeden}, \ref{jedendojeden} and \ref{punktkwadryke} we may consider the set $E\subset \p^3$ such that $\dim E\leq 1$ and the variety $S$ gives a correspondence of $BS(\textit W_P)\backslash (L\cup \sing(BS(\textit W_P)))$ and $\p^3\backslash E$ that with each non-degenerated quadric associates 2 points and with each singular quadric exactly one point.
This gives a birational map between  $BS(\textit W_P)$ and a double cover of $\p^3$ branched along the determinant surface corresponding to degenerated quadrics.
\kdow

\uwa
The theorem \ref{glowne} does not hold for a generic intersection of four quadrics and a small resolution of the corresponding double cover. Although the Euler characteristics are equal, the Calabi-Yau manifolds are not birational. On the complete intersection, due to the Picard-Lefschetz theorem, every divisor has self intersection at least 16. On the small resolution of the double cover there is a divisor that has self intersection 2, namely the pullback of the hyperplane section of $\p^3$.
\kuwa
\uwa\label{euler2}
Using the theorem \ref{glowne} and \cite{Batyrev} the Hodge numbers of the generic intersection containing a plane and small resolution of a corresponding double cover are equal. The second ones were computed in \ref{hodge}.
\kuwa
\section*{Acknowledgements}
I would like to thank very much S. Cynk for introducing me to the subject and guiding me through it. I am also grateful for many useful ideas.

I thank L. Bonavero for teaching me a lot about K3 surfaces, showing me easier proofs and particularly for finding and correcting many mistakes.

I also owe thanks to G. Kapustka and M. Kapustka for interesting talks about Calabi-Yau varieties.

\vskip10pt

\end{document}